\documentclass[12pt]{amsart}
\usepackage{epsfig}
\usepackage{mathtools}
\usepackage{amssymb}
\usepackage{amsmath}
\usepackage{amsthm}
\usepackage{aliascnt}
\usepackage{indentfirst}
\usepackage{footnpag}
\usepackage{bm}
\usepackage{graphicx,pgf}
\usepackage{tikz}
\allowdisplaybreaks
\numberwithin{equation}{section}
\theoremstyle{definition}
\newtheorem{theorem}{Theorem}[section]
\newtheorem{lemma}[theorem]{Lemma}
\newtheorem{*theorem}[theorem]{$\bullet$ Theorem}
\newtheorem{proposition}[theorem]{Proposition}

\newtheorem{corollary}[theorem]{Corollary}
\newtheorem{*corollary}[theorem]{$\bullet$ Corollary}

\newtheorem*{definition*}{Definition}

\theoremstyle{remark}
\newcounter{rem}
\newtheorem{remark}[rem]{Remark}
\newtheorem{remark*}[rem]{$\star$ Remark}
\usepackage[numbers]{natbib}
\newcommand{\norm}[1]{\lVert#1\rVert}
\newcommand{\abs}[1]{\lvert#1\rvert}

\title{{\bf Non-axially symmetric solutions of a mean field equation on $\mathbb{S}^2$}}
\author{Changfeng Gui}
\address{Department of Mathematics, University of Texas, San Antonio, TX, 78249}
\email{\tt changfeng.gui@utsa.edu}
\author{Yeyao Hu}
 \address{Department of Mathematics,
    University of Texas, San Antonio, TX 78249}
\email{\tt yeyao.hu@utsa.edu}
                  
\date{}
\begin{document}
\maketitle

\begin{quote}
{ \textbf{Abstract}}: We prove the existence of a family of blow-up solutions of a mean field equation on sphere. The solutions blow up at four points where the minimum value of a potential energy function (involving the Green's function) is attained. The four blow-up points form a regular tetrahedron. Moreover, the solutions we build have a group of symmetry $T_d$ which is isomorphic to the symmetric group $S_4$.  Other families of solutions can be similarly constructed with blow-up points at the vertices of equilateral triangles on a great circle or  other inscribed platonic solids (cubes, octahedrons, icosahedrons and dodecahedrons).  All of these solutions have the symmetries of the corresponding configuration, while they are non-axially symmetric.  


\end{quote}


\section{Introduction}

\noindent In this paper, we consider a mean field equation on sphere $\mathbb{S}^2$, i.e.
\begin{equation}\label{eq:mfeS2}
\Delta_g u + \rho \left(\frac{e^u}{\int_{\mathbb{S}^2}e^u}-\frac{1}{4\pi}\right)=0,
\end{equation}
where $\Delta_g$ stands for the Laplace-Beltrami operator on $\mathbb{S}^2$ associated to the metric $g$ inherited from the ambient Euclidean metric. 

One may consider the more general form of the equation on a compact Riemannian surface $M$ without boundary:
\begin{equation}\label{eq:mfeg}
\Delta u + \rho\left(\frac{h e^u}{\int_M h e^u}-\frac{1}{\abs{M}}\right)=0,
\end{equation}
where $h\in C^{\infty}(M)$ is a positive potential function and $\abs{M}$ is the total area of the surface $M$.


Li in \cite{li1999harnack} proved that the solutions of (\ref{eq:mfeg}) are uniformly bounded if $\rho$ varies on any compact subset of $\mathbb{R}^{+}\setminus 8\pi\mathbb{N}$. Hence, the blow-up phenomena could only occur when $\rho\rightarrow 8\pi m$ where $m\in\mathbb{N}$. To study the general existence result of the mean field equation, he initiated a program to compute the topological degree $d_{\rho}$ of a map related to (\ref{eq:mfeg}). He also showed that $d_{\rho}=1$ as long as $\rho<8\pi$.  Due to the result of Li and the homotopy invariance of the degree, it is readily checked that $d_{\rho}$ is a constant on the interval $(8\pi(m-1),8\pi m)$ and is independent of $h$ and the metric of $M$. In particular, Lin in \cite{lin2000} calculated the topological degree of (\ref{eq:mfeS2}): $d_{\rho}=-1$ when $8\pi<\rho<16\pi$, and $d_{\rho}=0$ when $16\pi<\rho<24\pi$. Chen and Lin later in \cite{CPA:CPA3014} proved a priori bound for a sequence $\rho_n$ with $\rho=\rho_n$ in (\ref{eq:mfeg}). Using this a priori bound, they were able to calculate the degree $d_{\rho}$ in \cite{CPA:CPA10107}: $d_{\rho}=C_{m-\chi(M)}^m$ with $\rho\in(8\pi m, 8\pi (m+1))$ where $m\in\mathbb{N}$ and $\chi(M)$ denotes the Euler characteristic of the Riemann surface $M$. They evaluated the jump of degree across $8\pi m$ by calculating the degree contributed by blow-up solutions. In the case of (\ref{eq:mfeS2}), we have $\chi(\mathbb{S}^2)=2$ so that the degree $d_{\rho}=0$ for $\rho\in (8\pi m,8\pi(m+1))$ with $m\geq 2$. Therefore, it is not clear whether there exist blow-up solutions for (\ref{eq:mfeS2}) as $\rho\rightarrow 8\pi m$ with $m\geq 3$ solely by the result of Chen and Lin. However, Lin in \cite{lin2000} did establish the existence of the blow-up solutions to (\ref{eq:mfeS2}) when $\rho$ approaches $16\pi$ from above. Moreover, he showed that any sequence of axially symmetric nontrivial solutions must blow up at two points antipodal to each other.

Concerning the uniqueness of solution to (\ref{eq:mfeS2}), Lin in \cite{lin2000uniqueness} showed that the solution to (\ref{eq:mfeS2}) is unique for $0<\rho<8\pi$. In other words, (\ref{eq:mfeS2}) only admits constant solutions for $0<\rho<8\pi$. Most recently, the first author and Moradifam \cite{gui2016sphere} developed a new tool named ``sphere covering inequality" to extend the uniqueness result to a broader parameter range $\rho\in(0,8\pi)\cup(8\pi,16\pi]$ when the solutions of (\ref{eq:mfeS2}) that have center of mass at origin are considered. In \cite{DOLBEAULT20092870}, the multiplicity of axially symmetric nontrivial solutions of (\ref{eq:mfeS2}) is carefully investigated by Dolbeault, Esteban and Tarantello. However, even the existence of non-axially symmetric solutions remains open. Our paper gives an affirmative answer to it. Interested reader is referred to the survey \cite{Tarantello2010931} for more details of mean field equations on a closed surface.

In this paper, we will construct blow-up solutions to (\ref{eq:mfeS2}) with blow-up 
 points forming  regular configurations, i.e., the vertices of equilateral triangles on a great circle or   inscribed platonic solids (tetrahedrons, cubes, octahedrons, icosahedrons and dodecahedrons). Morevoer,  these solutions posses the corresponding  symmetries of the configuration. 

 To make the construction easier to understand,  we will consider
  $\rho\rightarrow 32\pi$ and focus on a configuration of tetrahedon. 
  The solutions we construct blow up at exactly four points $\xi_1$, $\xi_2$, $\xi_3$ and $\xi_4$. The four points $(\xi_1,\cdots,\xi_4)$ form a regular tetrahedron.
   Furthermore, $(\xi_1,\cdots,\xi_4)$ appears to be a critical point of the function
\begin{equation*}
F(p_1,p_2,p_3,p_4)=4\pi\sum\limits_{j<k}G(p_j,p_k)
\end{equation*}
where $(p_1,p_2,p_3,p_4)\in (\mathbb{S}^2)^4$ and $G$ denotes the Green's function of $-\Delta_g$ which will be defined explicitly in Section \ref{sec2}. In \cite{CPA:CPA10107}, Chen and Lin defined a more general function $f_h$ (see (1.18) in \citep{CPA:CPA10107}) which determines the locations of blow-up points. Our $F$ is a special case of $f_h$ when we take $h\equiv 1$ and $M=\mathbb{S}^2$. In \cite{CPA:CPA10107}, they choosed the potential $h$ wisely so that $f_h$ is a morse function to construct blow-up solutions of (\ref{eq:mfeg}). However, in our case $F$ is actually invariant under any orthogonal transformation due to the trait of the Green's function. Hence, their argument can not be applied here. To overcome the difficulty caused by the degeneracy of $F$, we assume further that the blow-up solutions posses the tetrahedral symmetry. More precisely, we will build the solutions among the class of functions that satisfy the following property:
\begin{equation}\label{symmetry}
u(y)=u(Ty), \textrm{ for all }T\in T_d \textrm{ and   any }y\in\mathbb{S}^2
\end{equation}
where $T_d$ is the group of symmetry of a regular tetrahedron. The group $T_d$ is isomorphic to the symmetric group $S_4$ since there is exactly one such symmetry for each permutation of the vertices of the tetrahedron. We now treat $T_d$ as a subgroup of order $24$ of the orthogonal group $O(3,\mathbb{R})$. Finally, by fixing the four blow-up points and assuming the tetrahedral symmetry, we are able to find blow-up solutions of (\ref{eq:mfeS2}) using the Lyapunov-type reduction.

We would like to point out that when $\rho \le 24 \pi$,  it is expected that all 
solutions must be axially symmetric (see, e.g., \cite{sun} for some partial 
results).  The construction in this paper (see Remark \ref{rmk1}) shows that $\rho=24 \pi$ is indeed the borderline value for the existence of non-axially symmetric solutions. 

The paper is organized as follows. In Section 2, we state the main result and some preliminaries. In Section 3, we construct our approximate solutions and get some useful estimates. Section 4 is devoted to the proof of the invertibility of the linearized operator. In section 5, we reduce the problem to a problem of finding the scale of bubbles. In Section 6, we solve the reduced problem, i.e. solve the scale $\lambda$ as long as the parameter $\rho$ is given.

\section{Main Result and Preliminaries}\label{sec2}
\noindent Before we state the main result, let us first introduce the Green's function $G(y,y')$ of $-\Delta_g$ of $\mathbb{S}^2$:
\begin{equation}\label{eq:green}
-\Delta_{g(y)} G(y,y')=\delta_{y'}(y) - \frac{1}{4\pi} \textrm{ in }\mathbb{S}^2,
\end{equation}
and
\begin{equation*}
\int_{\mathbb{S}^2}G(y,y')dH^2(y')=0 \textrm{ for all }y\in \mathbb{S}^2,
\end{equation*}
where $dH^2$ denotes the two-dimensional hausdorff measure.

In particular, we have the explicit formula of $G(y,y')$:
\begin{equation}\label{eq:regular}
G(y,y')=-\frac{1}{2\pi}\log{|y-y'|},
\end{equation}
where $|y-y'|$ denotes the euclidean distance of $y$ and $y'$ when $\mathbb{S}^2$ is embedded into $\mathbb{R}^3$ in the standard way.

Let $U_{\lambda,p}$ be the standard bubble at a point $p\in\mathbb{S}^2$, i.e. $U_{\lambda,p}$ solves (\ref{eq:mfeS2}) given that $\rho=8\pi$. We construct an isothermal coordinate system $x=(x_1,x_2)$ around $p$ by the stereographic projection $\Pi_p: \mathbb{S}^2\setminus -p\rightarrow \mathbb{R}^2$. Locally the Riemannian metric can be written in these coordinates:
\begin{equation*}
g=\frac{4}{(1+|x|^2)^2}\left(dx^2_1+dx^2_2\right).
\end{equation*}
The area element is given by
\begin{equation*}
d A =\frac{4}{(1+|x|^2)^2}dx_1dx_2.
\end{equation*}
One can also connect the Laplace-Beltrami operator with the usual laplacian operator on $\mathbb{R}^2$ through the following:
\begin{equation*}
\Delta_{\mathbb{R}^2}=\frac{4}{(1+|x|^2)^2}\Delta_g.
\end{equation*}
Furthermore, we have
\begin{equation}\label{eq:greeniso}
G(y,p)=-\frac{1}{4\pi}\ln{\left(\frac{4|x|^2}{(1+|x|^2)}\right)},
\end{equation}
and
\begin{equation}\label{eq:greeniso2}
G(y,y')=-\frac{1}{4\pi}\ln{\left(\frac{4|x-x'|^2}{(1+|x|^2)(1+|x'|^2)}\right)},
\end{equation}
where $x=\Pi_p(y)$ and $x'=\Pi_p(y')$.

Let $V_{\lambda}=\ln{\left(\frac{8\lambda^2}{(\lambda^2+|x|^2)^2}\right)}$ be the family of solutions of the Liouville equation on $\mathbb{R}^2$:
\begin{equation}\label{eq:lioue}
\Delta V_{\lambda} + e^{V_{\lambda}}=0. 
\end{equation} 

If we assume further that
\begin{equation}\label{eq:massb}
\int_{\mathbb{S}^2}e^{U_{\lambda,p}}=8\pi,
\end{equation}
then we can write
\begin{equation}\label{eq:stardb}
U_{\lambda,p}(y)=V_{\lambda}(x) + 2\ln{(1+|x|^2)}-\ln{4}.
\end{equation}

Now we state our main result.

\begin{theorem}\label{th:mainresult}
Let $\epsilon\in (0,\epsilon_0)$ for some $\epsilon_0$ small enough. Let $\rho=32\pi+\epsilon$. Assume that $\xi_1$, $\xi_2$, $\xi_3$ and $\xi_4$ form a regular tetrahedron. Then for each $\epsilon$, there exist a $\lambda>0$ and a solution $u_{\lambda}$ to the equation (\ref{eq:mfeS2}) such that 
\begin{equation*}
\epsilon=(384\pi^2+o(1))\lambda^2 \ln{\frac{1}{\lambda}},
\end{equation*}
\begin{equation*}
u_{\lambda}(\xi_j) \rightarrow \infty \textrm{ for }j=1,2,3,4, 
\end{equation*}
\begin{equation*}
u_{\lambda}(x)\rightarrow -\infty \textrm{ for all }x\in \mathbb{S}^2\setminus\{\xi_1,\xi_2,\xi_3,\xi_4\}\textrm{ as }\epsilon\rightarrow 0.
\end{equation*}
Moreover, $u_{\lambda}$ possesses tetrahedral symmetry, i.e.
\begin{equation*}
u_{\lambda}(y)=u_{\lambda}(Ty), \textrm{ for all }T\in T_d \textrm{ and   any }y\in\mathbb{S}^2,
\end{equation*}
and
\begin{equation*}
\frac{\rho}{\int_{\mathbb{S}^2}e^{u_{\lambda}}}e^{u_{\lambda}}\rightarrow 8\pi\sum\limits_{j=1}^4\delta_{\xi_j} \textrm{ in a sense of measure, as }\epsilon\rightarrow 0.
\end{equation*}
\end{theorem}
The proof of Theorem \ref{th:mainresult} relies on a Lyapunov-type reduction. We first construct the approximation solution which behaves like the standard bubble $U_{\lambda,\xi_j}$ near the blow-up point $\xi_j$ and behaves like the Green's function away from these four points. Then we carry out a finite dimensional variational reduction for which the main ingredient is an analysis, of independent interest, of bounded invertibility up to the dilations of the linearized operator in suitable $L^{\infty}-$weighted spaces with certain symmetries. The setting successfully reduce the original problem into a problem of finding the appropriate scale $\lambda$ of the bubbles.

\begin{remark}\label{rmk1}
The same type of construction also works for the case where the number of blow-up points are $m=3$, $m=6$, $m=8$, $m=12$ and $m=20$ respectively, with $\rho$ tending  $8m\pi$. To be more precise, it is possible to build blow-up solutions that concentrate at three points which make a equilateral triangle on the great circle, as $\rho \rightarrow 24 \pi$. We can also show the existence of blow-up solutions that blow up at exactly six points, as $\rho \rightarrow 48 \pi$. In this case, the six points form a regular octahedron. Note here, $(\xi_1,\xi_2,\xi_3)$ is a critical point of the function $F(p_1,p_2,p_3)=\sum_{1 \leq j<k\leq 3} G(p_j,p_k)$ if $(\xi_1,\xi_2,\xi_3)$ forms a equilateral triangle on the great circle; while $(\xi_1,\cdots,\xi_6)$ is a critical point of the function $F(p_1,\cdots,p_6)=\sum_{1 \leq j<k \leq 6}G(p_j,p_k)$ if $(\xi_1,\cdots,\xi_6)$ forms a regular octahedron. Similarly, the ``cubic" blow-up solutions (the solutions that blow up at eight points which form a cube, as $\rho \rightarrow 64 \pi$) can be found when $m=8$. It is no surprise that the ``cube" configuration is indeed a critical configuration of $F$ when $m=8$.  The ``icosahedral" blow-up solutions (the solutions that blow up at twelve points which form a regular icosahedron, as $\rho \rightarrow 96 \pi$) exist when $m=12$. Furthermore, the ``dodecahedral" solutions (the solutions that blow up at twenty points which form a regular dodecahedron as $\rho \rightarrow 160 \pi$) can be built in the same fashion when $m=20$.  The construction of these solutions could follow line by line the proof of Theorem \ref{th:mainresult} with suitable change of numbers, so we omit the details.  We also would like to point out that these solutions posses certain kinds of symmetries but they are not axially symmetric. 

\end{remark}

\begin{remark}\label{rmk2}
It is proved that the platonic solid configurations when $m=4, 6, 12$ minimize the corresponding $F$s (see \cite{kolushov1997extremal} for $m=6$ and see \cite{cohn2007universally} for $m=12$). For a rigorous proof of minimality of the tetrahedral configuration, one can refer to \cite{dragnev2002discrete} in which the authors also showed the optimality of a five point configuration. However, the ``cube" configuration is not a minimizing configuration of $F$ when $m=8$. The optimal configuration in this case is called a ``twisted cuboid" (see \cite{trinh2016dynamics}), consisting of two parallel rings containing a square, with the square shifted by $45^{\circ}$ between each ring. The minimality of the ``dodecahedral" configuration is unknown for the case $m=20$.
\end{remark}

\section{Approximate Solution}
\noindent In this section, we will construct the approximate solution of the equation (\ref{eq:mfeS2}) and obtain some estimates of this approximate solution. Let $R_0>0$ be a small fixed number. Let $\eta$ be a standard cut-off function such that
\begin{equation*}
\eta(s)=1 \textrm{ for }s\leq 1; \eta(s)=0 \textrm{ for }s\geq 2; 0<\eta(s)<1 \textrm{ for } 1<s<2.
\end{equation*}
We further assume that 
\begin{equation*}
\abs{\eta'(s)}\leq 2.
\end{equation*}
Let 
\begin{equation}\label{eq:cutoff1}
\eta_{t,\xi}(y)=\eta\left(\frac{|\Pi_{\xi}(y)|}{t}\right),
\end{equation} 
for any $\xi\in\mathbb{S}^2$ and $t>0$.

Given $\epsilon\in(0,\epsilon_0)$, we choose $\lambda>0$ such that
\begin{equation}\label{eq:rangepar}
192\pi\lambda^2\ln{\frac{1}{\lambda}}<\epsilon<768\pi\lambda^2\ln{\frac{1}{\lambda}}.
\end{equation}

In other words, the above inequality can also be written as
\begin{equation}\label{eq:parr2}
\lambda_1(\epsilon)<\lambda<\lambda_2(\epsilon),
\end{equation}
where one can solve $\lambda_1(\epsilon)$ and $\lambda_2(\epsilon)$ from (\ref{eq:rangepar}).

Let $w_{\lambda,k}$ be the solution of the following equation:
\begin{equation}\label{eq:bubble1}
-\Delta_g w_{\lambda,k}= e^{U_{\lambda,\xi_k}}\eta_{R_0,\xi_k}-m_0,
\end{equation}
\begin{equation*}
\int_{\mathbb{S}^2} w_{\lambda,k}=0,
\end{equation*}
where
\begin{equation}\label{eq:mass1}
4\pi m_0=\int_{\mathbb{S}^2}e^{U_{\lambda,\xi_k}}\eta_{R_0,\xi_k},
\end{equation}
for $k=1,2,3,4$.

By simple calculations, one can obtain the so-called ``mass" of $w_{\lambda,k}$, i.e.
\begin{equation}\label{eq:mass}
m_0=2+O(\lambda^2).
\end{equation}

We introduce $\tilde{w}_{\lambda}$ to be the sum of $w_{\lambda,k}$, i.e.
\begin{equation}\label{eq:avg0}
\tilde{w}_{\lambda}=\sum\limits_{k=1}^4 w_{\lambda,k},
\end{equation}
and a constant related to $\lambda$
\begin{equation}\label{eq:avg}
\overline{w}_{\lambda}=2\ln{\lambda}+5\ln{2}-4\pi \sum\limits_{j< k}G(\xi_j,\xi_k).
\end{equation}

Then we are ready to provide an ansatz for solutions of the equation (\ref{eq:mfeS2}):
\begin{equation}\label{eq:approx}
w_{\lambda}=\tilde{w}_{\lambda}+\overline{w}_{\lambda}.
\end{equation}

Let us then calculate the values of $w_{\lambda,k}$ at the blow-up points $\xi_k$:
\begin{equation}\label{eq:center}
w_{\lambda,k}(\xi_k)=\int_{\mathbb{S}^2} G(\xi_k,y)\left[e^{U_{\lambda,\xi_k}}\eta_{R_0,\xi_k}(y)-m_0\right]dH^2(y)
\end{equation}
\begin{equation*}
=\int_{B(0,R_0)} -\frac{1}{4\pi}\ln{\left(\frac{4|x|^2}{(1+|x|^2)}\right)}\frac{8\lambda^2}{(\lambda^2+|x|^2)^2}dx+O(\lambda^2)
\end{equation*}
\begin{equation*}
=\int_{B(0,\frac{R_0}{\lambda})}\left[-\frac{1}{2\pi}\ln{2}-\frac{1}{2\pi}\ln{\lambda}-\frac{1}{2\pi}\ln{|z|}+\frac{1}{4\pi}\ln{(1+\lambda^2|z|^2)}\right]\frac{8}{(1+|z|^2)^2}dz + O(\lambda^2)
\end{equation*}
\begin{equation*}
=-4\ln{2}-4\ln{\lambda}+\frac{4\lambda^2\ln{\lambda}}{\lambda^2+R^2_0}-\frac{4}{\pi}\int_{B(0,\frac{R_0}{\lambda})} \frac{\ln{|z|}}{(1+|z|^2)^2}dz 
\end{equation*}
\begin{equation*}
+ \frac{2}{\pi} \int_{B(0,\frac{R_0}{\lambda})} \frac{\ln{(1+\lambda^2|z|^2)}}{(1+|z|^2)^2}dz+ O(\lambda^2)
\end{equation*}
\begin{equation*}
=-4\ln{\lambda} -4\ln{2} -4\lambda^2 \ln{\lambda}    + O(\lambda^2),
\end{equation*}
where $x=\Pi_{\xi_k}(y)=\lambda z$.

For $\abs{z}\leq\frac{R_0}{\lambda}$, we have
\begin{equation}\label{eq:inner}
w_{\lambda,k}(y)-w_{\lambda,k}(\xi_k)=\int_{\mathbb{S}^2}\left[G(y,y')-G(\xi_k,y')\right]e^{U_{\lambda,\xi_k}}\eta_{R_0,\xi_k}dH^2(y')
\end{equation}
\begin{equation*}
=\int_{B(0,\frac{R_0}{\lambda})}-\frac{1}{2\pi}\left[\ln{|z-z'|}-\ln{|z'|}\right]\frac{8}{(1+|z'|^2)^2}dz'
\end{equation*}
\begin{equation*}
+\frac{2}{\pi}\int_{B(0,\frac{R_0}{\lambda})}\frac{\ln{(1+\lambda^2|z|^2)}}{(1+|z'|^2)^2}dz'+ O(\lambda^2) + O(\lambda^3|z|)
\end{equation*}
\begin{equation*}
=\ln{\left(\frac{1}{(1+|z|^2)^2}\right)}+2\ln{(1+\lambda^2|z|^2)}+\lambda^2 f_k(x)+O(\lambda^3|z|),
\end{equation*}
where $x'=\Pi_{\xi_k}(y')=\lambda z'$ and $f_k$ is a smooth function of $x$ which is uniformly bounded with respect to $\lambda$ .

For $\abs{z}\geq\frac{2 R_0}{\lambda}$, i.e. $|x|\geq 2 R_0$, we have
\begin{equation}\label{eq:outer}
w_{\lambda,k}(y)=\int_{\mathbb{S}^2} G(y,y')e^{U_{\lambda,\xi_k}}\eta_{R_0,\xi_k}dH^2(y')
\end{equation}
\begin{equation*}
=\int_{B(0,\frac{R_0}{\lambda})} G(y,\Pi^{-1}_{\xi_k}(\lambda z'))\frac{8}{(1+|z'|)^2}dz'+ O(\lambda^2)
\end{equation*}
\begin{equation*}
=\int_{B(0,\frac{R_0}{\lambda})} \left[G(y,\xi_k)+ \lambda\nabla_{x'} G(y,\xi_k)\cdot z' +\frac{\lambda^2 (z'\nabla_{x'}^2 G(y,\xi_k){z'}^{T})}{2}\right]\frac{8}{(1+|z'|^2)^2}dz'
\end{equation*}
\begin{equation*}
+O(\lambda^2)
\end{equation*}
\begin{equation*}
=8\pi G(y,\xi_k)+4\pi\lambda^2 \int_0^{\frac{R_0}{\lambda}}\textrm{Tr}(\nabla_{x'}^2 G(y,\xi_k))\frac{r^3}{(1+r^2)^2}dr+ O(\lambda^2) 
\end{equation*}
\begin{equation*}
=8\pi G(y,\xi_k)-4\lambda^2 \ln{\lambda}+\lambda^2 \tilde{f}_k(y),
\end{equation*}
where $\tilde{f}_k$ is a smooth function of $y$ which is uniformly bounded with respect to $\lambda$.

In particular, we can get
\begin{equation}\label{eq:peak}
w_{\lambda,k}(\xi_j)=8\pi G(\xi_j,\xi_k)-4\lambda^2 \ln{\lambda}+O(\lambda^2),
\end{equation}
for $j\neq k$ and $j,k=1,2,3,4$. 

Combining (\ref{eq:center}) and (\ref{eq:inner}), we have for $|z|<\frac{R_0}{\lambda}$,
\begin{equation}\label{eq:inner1}
w_{\lambda,k}(\Pi^{-1}_{\xi_k}(\lambda z))=-4\ln{\lambda}-4\ln{2}-4\lambda^2 \ln{\lambda}+\ln{\left(\frac{1}{(1+|z|^2)^2}\right)} + 2\ln{(1+\lambda^2|z|^2)}
\end{equation}
\begin{equation*}
+\lambda^2 f_k(x)+O(\lambda^3|z|).
\end{equation*}
Here we abuse the notation $f_k$ a little bit to denote a smooth function of $x$ which is uniformly bounded with respect to $\lambda$.

To estimate the values of $w_{\lambda,k}$ in the annulus $\{R_0<|x|<2R_0\}$, we compare $w_{\lambda,k}$ with a function $W_{\lambda,k}$ constructed by gluing the inner approximation and the outer approximation together using an ``intermediate" layer $\eta_{\lambda^{\alpha},\xi_k}$ for some $\alpha\in(0,1)$.

Let 
\begin{equation}\label{eq:glueil}
W_{\lambda,k}=w_i\eta_{\lambda^{\alpha},\xi_k}+w_o(1-\eta_{\lambda^{\alpha},\xi_k}),
\end{equation}
where
\begin{equation}\label{eq:gluein}
w_i(\Pi^{-1}_{\xi_k}(\lambda z))=-4\ln{\lambda}-4\ln{2}-4\lambda^2 \ln{\lambda}+\ln{\left(\frac{1}{(1+|z|^2)^2}\right)} + 2\ln{(1+\lambda^2|z|^2)},
\end{equation}
and
\begin{equation}\label{eq:glueou}
w_o(y)=8\pi G(y,\xi_k)-4\lambda^2 \ln{\lambda}.
\end{equation}

We have the following lemma:
\begin{lemma}\label{lm:gluelm}
For some $\alpha\in(0,1)$, there exist a constant $C>0$ independent of $\lambda$ and a $\alpha'\in(0,1)$ such that
\begin{equation*}
\norm{w_{\lambda,k}-W_{\lambda,k}}_{\infty}\leq C \lambda^{\alpha'}.
\end{equation*}
\end{lemma}
\begin{proof}
It is easy to verify that
\begin{equation}\label{eq:center2}
(w_{\lambda,k}-W_{\lambda,k})(\xi_k)=O(\lambda^2).
\end{equation}
From (\ref{eq:glueil})-(\ref{eq:glueou}), we have
\begin{equation}\label{eq:lapglue}
-\Delta_g W_{\lambda,k}(y)=-\frac{(1+|x|^2)^2}{4}\Delta_{\mathbb{R}^2}W_{\lambda,k}(\Pi^{-1}_{\xi_k}(x))
\end{equation}
\begin{equation*}
=e^{U_{\lambda,\xi_k}}\eta_{\lambda^{\alpha},\xi_k}-2-\frac{(1+|x|^2)^2}{4}[2\nabla_x (w_i-w_o)\cdot \nabla_x \eta_{\lambda^{\alpha},\xi_k}+(w_i-w_o)\Delta_{\mathbb{R}^2}\eta_{\lambda^{\alpha},\xi_k}].
\end{equation*}
Let
\begin{equation*}
r_{k,\lambda}(x)=\frac{(1+|x|^2)^2}{4}[2\nabla_x (w_i-w_o)\cdot \nabla_x \eta_{\lambda^{\alpha},\xi_k}+(w_i-w_o)\Delta_{\mathbb{R}^2}\eta_{\lambda^{\alpha},\xi_k}].
\end{equation*}

One can show that
\begin{equation*}
\norm{r_{k,\lambda}}_{\infty}\leq C \lambda^{2-2\alpha}.
\end{equation*}

By (\ref{eq:bubble1}) and (\ref{eq:lapglue}), we have
\begin{eqnarray*}
-\Delta_g(w_{\lambda,k}-W_{\lambda,k})&=&e^{U_{\lambda,\xi_k}}(\eta_{R_0,\xi_k}-\eta_{\lambda^{\alpha},\xi_k})+ r_{k,\lambda}(x) + O(\lambda^2)\\
&=&\tilde{r}_{k,\lambda}(x).
\end{eqnarray*}

Let $\tilde{W}_{\lambda,k}$ be the unique solution of the problem
\begin{equation}\label{eq:normglue}
-\Delta_g \tilde{W}_{\lambda,k}=\tilde{r}_{k,\lambda}(x),
\end{equation}
\begin{equation*}
\int_{\mathbb{S}^2}\tilde{W}_{\lambda,k}=0.
\end{equation*}

By the elliptic regularity estimate and the Poincare's inequality, we have
\begin{equation}\label{eq:ellest}
\norm{\tilde{W}_{\lambda,k}}_{H^2(\mathbb{S}^2)}\leq C\norm{\tilde{r}_{k,\lambda}}_{L^2(\mathbb{S}^2)}.
\end{equation}
It is readily checked that
\begin{equation}\label{eq:estl2}
\norm{\tilde{r}_{k,\lambda}}_{L^2(\mathbb{S}^2)}\leq C\lambda^{2-3\alpha}.
\end{equation}
We can repeat a similar calculation as we did in (\ref{eq:center}) to derive
\begin{equation}\label{eq:center3}
|\tilde{W}_{\lambda,k}(\xi_k)|\leq C(\lambda^{2-2\alpha}+\lambda^{2\alpha})\ln{\frac{1}{\lambda}}.
\end{equation} 

Choose $\alpha$ appropriately and combine (\ref{eq:center2}), (\ref{eq:ellest}), (\ref{eq:estl2}) and (\ref{eq:center3}), we have
\begin{equation*}
\norm{w_{\lambda,k}-W_{\lambda,k}}_{\infty}\leq C\lambda^{\alpha'},
\end{equation*}
for some $\alpha'\in(0,1)$.
\end{proof}

 We also obtain the following lemma concerning the values of $w_{\lambda}$ near the blow-up points $\xi_k$:
\begin{lemma}\label{lm:inner}
We have the inner approximation of $w_{\lambda}$ inside the ball $z\in B(0,\frac{R_0}{\lambda})$,
\begin{equation}\label{eq:innersum}
w_{\lambda}(\Pi^{-1}_{\xi_k}(\lambda z))=\ln{\left(\frac{8}{\lambda^2(1+|z|^2)^2}\right)} -16\lambda^2\ln{\lambda}+2\ln{(1+\lambda^2|z|^2)}-\ln{4}
\end{equation}
\begin{equation*}
+4\pi\lambda^2 \sum\limits_{j,j\neq k} z\nabla^2_x G(\Pi^{-1}_{\xi_k}(x),\xi_j)\vert_{x=0} z^T+\lambda^2 f(x)+ O(\lambda^3 |z|^3)+O(\lambda^3|z|),
\end{equation*}
for $k=1,2,3,4$.
Here $f(x)$ is a smooth function which is uniformly bounded with respect to $\lambda$.
\end{lemma}
\begin{remark}\label{rmk:3}
Note that here we use the fact that
\begin{equation*}
\sum\limits_{j,j\neq k}\nabla_x G(\Pi^{-1}_{\xi_k}(x),\xi_j)\vert_{x=0}=0.
\end{equation*}
In other words, $(\xi_1,\xi_2,\xi_3,\xi_4)$ is a critical point of the function
\begin{equation*}
F(p_1,p_2,p_3,p_4)=4\pi\sum\limits_{j< k}G(p_j,p_k).
\end{equation*}
\end{remark}

We then give the outer approximation:
\begin{lemma}\label{lm:outer}
When $\min\limits_{k=1,2,3,4}{|\Pi_{\xi_k}(y)|}\geq 2 R_0$, we have
\begin{equation}\label{eq:outersum}
w_{\lambda}(y)=2\ln{\lambda}+5\ln{2}-4\pi\sum\limits_{j< k} G(\xi_j,\xi_k)
\end{equation}
\begin{equation*}
+8\pi\sum\limits_{j=1}^4 G(y,\xi_j)- 16\lambda^2\ln{\lambda}+\lambda^2\tilde{f}(y),
\end{equation*}
where $\tilde{f}(y)$ is a smooth function of $y$ which is uniformly bounded with respect to $\lambda$.
\end{lemma}



From the previous three lemmas, we can estimate the $e^{w_{\lambda}}$. In particular, we have
\begin{equation}\label{eq:exp}
e^{w_{\lambda}}\leq \sum\limits_{k=1}^4 e^{U_{\lambda,\xi_k}}\left[1+\theta_{\lambda}(y)\right],
\end{equation}
where $\theta_{\lambda}$ is uniformly bounded with respect to $y$ and $\lambda$ and has the property that for some constant $C>0$,
\begin{equation*}
|\theta_{\lambda}(y)|\leq C \lambda\sum\limits_{k=1}^4 \left[\frac{|\Pi_{\xi_k}(y)|}{\lambda}+1\right].
\end{equation*}
                                                                                                                                                                                                                                                                                                                                                                                                                                           
More precisely, when $|z|<\frac{R_0}{\lambda}$, we have
\begin{equation}\label{eq:expinner}
e^{w_{\lambda}(\Pi^{-1}_{\xi_k}(\lambda z))}=e^{U_{\lambda,\xi_k}}\big[1+4\pi\lambda^2 \sum\limits_{j,j\neq k} z\nabla^2_x G(\Pi^{-1}_{\xi_k}(x),\xi_j)\vert_{x=0} z^T
\end{equation}
\begin{equation*}
 -16\lambda^2\ln{\lambda}+O(\lambda^2)+O(\lambda^3|z|^3)+O(\lambda^3 |z|)\big].
\end{equation*}
When $|\Pi_{\xi_k}(y)|\geq R_0$ for $k=1,2,3,4$, we have
\begin{equation}\label{eq:expouter}
e^{w_{\lambda}(y)}=O(\lambda^2).
\end{equation}

Let us then estimate the error of the approximate solution by inserting the ansats $w_{\lambda}$ into the equation (\ref{eq:mfeS2}).
\begin{lemma}\label{lm:error}
Let $S_{\rho}(u)=\Delta_g u + \rho\left(\frac{e^u}{\int_{\mathbb{S}^2}e^u}-\frac{1}{4\pi}\right)$.  Then there exists a constant $C>0$ such that
\begin{equation*}
|S_{\rho}(w_{\lambda})(\Pi^{-1}_{\xi_k}(\lambda z))|\leq C\left[\lambda^2\ln{\frac{1}{\lambda}}+\frac{\ln{\frac{1}{\lambda}}}{(1+|z|^2)^2}+\frac{|z|^2}{(1+|z|^2)^2}\right],
\end{equation*}
for $|z|<\frac{R_0}{\lambda}$ and $k=1,2,3,4$,
and
\begin{equation*}
|S_{\rho}(w_{\lambda})(y)|\leq C \lambda^2\ln{\frac{1}{\lambda}},
\end{equation*}
when $|\Pi_{\xi_k}(y)|\geq R_0$ for all $k=1,2,3,4$.

Furthermore, we also have $S_{\rho}(w_{\lambda})$ is invariant under orthogonal transformations that belong to the symmetry group $T_d$ of the regular tetrahedron.
\end{lemma}
\begin{proof}
We first use (\ref{eq:expinner}) and (\ref{eq:expouter}) to estimate the integral of $e^{w_{\lambda}}$, i.e.
\begin{equation}\label{eq:expint}
\int_{\mathbb{S}^2}e^{w_{\lambda}}=4\int_{B(\xi_k,R_0)}e^{w_{\lambda}}+\int_{\mathbb{S}^2\setminus(\bigcup\limits_{k=1}^4 B(\xi_k,R_0))}e^{w_{\lambda}}
\end{equation}
\begin{equation*}
=4 \int_{B(0,\frac{R_0}{\lambda})}\frac{8}{(1+|z|^2)^2}\big[1+4\pi\lambda^2 \sum\limits_{j,j\neq k} z\nabla^2_x G(\Pi^{-1}_{\xi_k}(x),\xi_j)\vert_{x=0} z^T
\end{equation*}
\begin{equation*}
 -16\lambda^2\ln{\lambda}+O(\lambda^2)+O(\lambda^3|z|^3)+O(\lambda^3 |z|)\big].
\end{equation*}
\begin{equation*}
+\int_{\mathbb{S}^2\setminus(\bigcup\limits_{k=1}^4 B(\xi_k,R_0))}e^{w_{\lambda}}
\end{equation*}
\begin{equation*}
=32\pi-896\pi\lambda^2\ln{\lambda}+O(\lambda^2).
\end{equation*}

When $|z|<\frac{R_0}{\lambda}$, we have
\begin{equation*}
S_{\rho}(w_{\lambda})(\Pi^{-1}_{\xi_k}(\lambda z))=\Delta_g w_{\lambda}(\Pi^{-1}_{\xi_k}(\lambda z))+\rho\left(\frac{e^{w_{\lambda}(\Pi^{-1}_{\xi_k}(\lambda z))}}{32\pi-896\pi\lambda^2\ln{\lambda}+O(\lambda^2)}-\frac{1}{4\pi}\right)
\end{equation*}
\begin{equation*}
=8+O(\lambda^2)-e^{U_{\lambda,\xi_k}}+\frac{(32\pi+\epsilon)e^{w_{\lambda}(\Pi^{-1}_{\xi_k}(\lambda z))}}{32\pi-896\pi\lambda^2\ln{\lambda}+O(\lambda^2)}-\frac{32\pi+\epsilon}{4\pi}
\end{equation*}
\begin{equation*}
=-\frac{\epsilon}{4\pi} + O(\lambda^2)+\frac{\epsilon+896\pi\lambda^2\ln{\lambda}+O(\lambda^2)}{32\pi-896\pi\lambda^2\ln{\lambda}+O(\lambda^2)}\cdot e^{U_{\lambda,\xi_k}}
\end{equation*}
\begin{equation*}
+O\left(\frac{\ln{\frac{1}{\lambda}}}{(1+|z|^2)^2}\right)+O\left(\frac{|z|^2}{(1+|z|^2)^2}\right).
\end{equation*}
We know from (\ref{eq:rangepar}) that $\epsilon=O(\lambda^2\ln{\lambda})$, then we have 
\begin{equation*}
|S_{\rho}(w_{\lambda})(\Pi^{-1}_{\xi_k}(\lambda z))|\leq C\left[\lambda^2\ln{\frac{1}{\lambda}}+\frac{\ln{\frac{1}{\lambda}}}{(1+|z|^2)^2}+\frac{|z|^2}{(1+|z|^2)^2}\right]
\end{equation*}
for $|z|<\frac{R_0}{\lambda}$ and $k=1,2,3,4$.

Similarly, we can estimate the outer error using (\ref{eq:expint}):
\begin{equation*}
S_{\rho}(w_{\lambda})(y)=-\frac{\epsilon}{4\pi}+O(\lambda^2) +\frac{32\pi+\epsilon}{32\pi+O(\lambda^2\ln{\lambda})}O(\lambda^2)
\end{equation*}
since (\ref{eq:expouter}) holds for all $|\Pi_{\xi_k}(y)|\geq R_0$ and $k=1,2,3,4$.

The rest of the lemma follows from the last identity.
\end{proof}

The equation (\ref{eq:mfeS2}) has a variational structure, i.e. critical points of the energy functional
\begin{equation}\label{eq:energy}
J_{\rho}(u)=\frac{1}{2}\int_{\mathbb{S}^2}|\nabla u|^2-\rho\ln{\left(\int_{\mathbb{S}^2}e^{u}\right)}+\frac{\rho}{4\pi}\int_{\mathbb{S}^2} u
\end{equation}
correspond to the solutions of the equation (\ref{eq:mfeS2}). Our next goal is to estimate the energy functional of the approximate solution $w_{\lambda}$.
\begin{lemma}\label{lm:energyest}
The energy of $w_{\lambda}$ is 
\begin{equation*}
J_{\rho}(w_{\lambda})=-64\pi^2 \sum\limits_{j< k} G(\xi_j,\xi_k)-32\pi\ln{(4\pi)}+2\epsilon \ln{\lambda}
\end{equation*}
\begin{equation*}
+384\pi \lambda^2\ln{\lambda}-\epsilon\left(\ln{(\pi)}-4\pi \sum\limits_{j< k}G(\xi_j,\xi_k)\right)+O(\lambda^2).
\end{equation*}
\end{lemma}
\begin{proof}
From (\ref{eq:expint}), we can compute
\begin{equation}\label{eq:energy2}
-\rho\ln{\left(\int_{\mathbb{S}^2}e^{w_{\lambda}}\right)}=-(32\pi+\epsilon)\ln{(32\pi-896\pi\lambda^2\ln{\lambda}+O(\lambda^2))}
\end{equation}
\begin{equation*}
=-32\pi (\ln{(32\pi)}-28\lambda^2\ln{\lambda}+O(\lambda^2))-\epsilon (\ln{(32\pi)}-28\lambda^2\ln{\lambda}+O(\lambda^2)).
\end{equation*}

Also, we can easily compute
\begin{equation}\label{eq:energy3}
\frac{\rho}{4\pi}\int_{\mathbb{S}^2}w_{\lambda}=(32\pi+\epsilon)\overline{w}_{\lambda}
\end{equation}
\begin{equation*}
=32\pi(2\ln{\lambda}+5\ln{2}-4\pi\sum\limits_{j< k}G(\xi_j,\xi_k))+\epsilon(2\ln{\lambda}+5\ln{2}-4\pi\sum\limits_{j< k} G(\xi_j,\xi_k)).
\end{equation*}
Then, the only term remaining is the following
\begin{equation}\label{eq:energy1}
\frac{1}{2}\int_{\mathbb{S}^2}|\nabla w_{\lambda}|^2=\frac{1}{2}\langle -\Delta_g w_{\lambda}, w_{\lambda}\rangle=\frac{1}{2}\langle -\Delta_g w_{\lambda}, \tilde{w}_{\lambda}\rangle
\end{equation}
\begin{equation*}
=\frac{1}{2}\int_{\mathbb{S}^2}\sum\limits_{k=1}^4 e^{U_{\lambda,\xi_k}}\eta_{R_0,\xi_k} \cdot \sum\limits_{k=1}^4 w_{\lambda,k}
\end{equation*}
\begin{equation*}
=2I_{1,1}+\frac{1}{2}\sum\limits_{j,k,j\neq k}I_{j,k},
\end{equation*}
where
\begin{equation*}
I_{1,1}=\int_{\mathbb{S}^2} e^{U_{\lambda,\xi_1}}\eta_{R_0,\xi_1}w_{\lambda,1},
\end{equation*}
and
\begin{equation*}
I_{j,k}=\int_{\mathbb{S}^2}e^{U_{\lambda,\xi_j}}\eta_{R_0,\xi_j}w_{\lambda,k},
\end{equation*}
for $j\neq k$.

Let us use (\ref{eq:inner1}) to compute $I_{1,1}$ first
\begin{equation*}
I_{1,1}=\int_{B(0,\frac{R_0}{\lambda})}\frac{8 w_i(\Pi^{-1}_{\xi_1}(\lambda z))}{(1+|z|^2)^2}dz+O(\lambda^2)
\end{equation*}
\begin{equation*}
=\left(-4\ln{\lambda}-4\ln{2}-4\lambda^2\ln{\lambda}\right)\int_{B(0,\frac{R_0}{\lambda})}\frac{8}{(1+|z|^2)^2}dz
\end{equation*}
\begin{equation*}
-16\int_{B(0,\frac{R_0}{\lambda})}\frac{\ln{(1+|z|^2)}}{(1+|z|^2)^2}dz+16\int_{B(0,\frac{R_0}{\lambda})}\frac{\ln{(1+\lambda^2|z|^2)}}{(1+|z|^2)^2}dz+O(\lambda^2)
\end{equation*}
\begin{equation*}
=-32\pi\ln{\lambda}-32\pi\ln{2}-64\pi\lambda^2\ln{\lambda}+O(\lambda^2).
\end{equation*}

Then, we use (\ref{eq:outer}) to compute $I_{j,k}$
\begin{equation*}
I_{j,k}=\int_{B(0,\frac{R_0}{\lambda})}\frac{8w_{\lambda,k}(\Pi^{-1}_{\xi_j}(\lambda z))}{(1+|z|^2)^2}dz+O(\lambda^2)
\end{equation*}
\begin{equation*}
=\int_{B(0,\frac{R_0}{\lambda})}\frac{8(8\pi G(\Pi^{-1}_{\xi_j}(\lambda z),\xi_k)-4\lambda^2\ln{\lambda}+\lambda^2 \tilde{f}_k(\Pi^{-1}_{\xi_j}(\lambda z))}{(1+|z|^2)^2}dz + O(\lambda^2)
\end{equation*}
\begin{equation*}
=8\pi G(\xi_j,\xi_k)\int_{B(0,\frac{R_0}{\lambda})}\frac{8}{(1+|z|^2)^2}dz+32\pi\lambda^2\int_{B(0,\frac{R_0}{\lambda})}\frac{z\nabla^2_x G(\Pi^{-1}_{\xi_j}(x),\xi_k)\vert_{x=0} z^T}{(1+|z|^2)^2}dz
\end{equation*}
\begin{equation*}
-32\pi \lambda^2\ln{\lambda}+O(\lambda^2)
\end{equation*}
\begin{equation*}
=64\pi^2 G(\xi_j,\xi_k)-64\pi\lambda^2\ln{\lambda}+O(\lambda^2).
\end{equation*}

Therefore, by (\ref{eq:energy2}), (\ref{eq:energy3}) and (\ref{eq:energy1}) we have
\begin{equation*}
J_{\rho}(w_{\lambda})=-64\pi^2 \sum\limits_{j< k} G(\xi_j,\xi_k)-32\pi\ln{(4\pi)}+2\epsilon \ln{\lambda}
\end{equation*}
\begin{equation*}
+384\pi \lambda^2\ln{\lambda}-\epsilon\left(\ln{(\pi)}-4\pi \sum\limits_{j< k}G(\xi_j,\xi_k)\right)+O(\lambda^2).
\end{equation*}
\end{proof}

We also need the following lemma on the dependence of $\lambda$ of the approximate solution $w_{\lambda}$ later in this paper:
\begin{lemma}\label{lm:deriv}
Inside the ball $x\in B(0,R_0)$, we have
\begin{equation*}
\frac{\partial w_{\lambda}}{\partial\lambda}(\Pi^{-1}_{\xi_k}(x))=\frac{2(|z|^2-1)}{\lambda(|z|^2+1)}+O(1),
\end{equation*}
where $x=\lambda z$.
When $\min\limits_{k=1,2,3,4}|\Pi_{\xi_k}(y)|\geq R_0$, we have
\begin{equation*}
\frac{\partial w_{\lambda}}{\partial\lambda}(y)=\frac{2}{\lambda}+O(1).
\end{equation*}
\end{lemma}
One can mimic the calculations we did for the derivation of Lemma \ref{lm:inner} and Lemma \ref{lm:outer} and follow the same idea we used in establishing Lemma \ref{lm:gluelm} to prove this lemma, we omit the details here for simplicity.

\section{The Linearized Operator}
\noindent In this section, we will establish a solvability theory for the linearized operator under suitable orthogonality condition.

Let us introduce an operator
\begin{equation}\label{eq:linear2}
\mathcal{L}(u)=\Delta_g u + \frac{\rho}{\int_{\mathbb{S}^2}e^{w_{\lambda}}}e^{w_{\lambda}}u.
\end{equation}

The above operator is connected with the linearized operator of $S_{\rho}$ through the following
\begin{equation}\label{eq:linear3}
S^{'}_{\rho}(w_{\lambda})(u)=\mathcal{L}\left(u-\frac{\int_{\mathbb{S}^2}e^{w_{\lambda}}u}{\int_{\mathbb{S}^2}e^{w_{\lambda}}}\right).
\end{equation}

Let
\begin{equation}\label{eq:linearscale}
L(u)=\lambda^2 \mathcal{L}(u).
\end{equation}

If we consider the isothermal coordinates at $\xi_k$ and blow up the sphere $\mathbb{S}^2$ by the scale $\lambda$ to $\mathbb{S}^2_{\lambda}$, then the linearized operator $\mathcal{L}$ scaled by $4\lambda^2$ formally approaches a linear operator $\tilde{L}$ in $\mathbb{R}^2$, i.e.
\begin{equation}\label{eq:linearlim}
\tilde{L}(u)=\Delta_{z} u +\frac{8}{(1+|z|^2)^2}u,
\end{equation}
where $z=\frac{\Pi_{\xi_k}(y)}{\lambda}$.

The operator $\tilde{L}$ can be obtained by linearizing the equation $\Delta u+ e^u=0$ at the radial solution $V_0(z)=\ln{\left(\frac{8}{(1+|z|^2)^2}\right)}$. An important fact we are going to employ in developing the solvability theory is the non-degeneracy of $V_0$ modulo the invariance of the equations under translations and dilations, i.e.
\begin{equation*}
\zeta\mapsto V_0(z-\zeta); \textrm{ }s\mapsto V_0(z/s)-2\ln{s}.
\end{equation*}
Thus we set,
\begin{equation*}
\varphi_k(z)=\frac{\partial}{\partial{\zeta_k}}V_0(z+\zeta)\rvert_{\zeta=0}, \textrm{ }i=1,2.
\end{equation*}
\begin{equation*}
\varphi_0(z)=\frac{\partial}{\partial{s}}\left[V_0(z/s)-2\ln{s}\right]\rvert_{s=1}.
\end{equation*}
Direct computation shows that 
\begin{equation*}
\varphi_k=\frac{-4z_k}{1+|z|^2},
\end{equation*}
for $k=1,2$ and
\begin{equation*}
\varphi_0=\frac{2(|z|^2-1)}{1+|z|^2}.
\end{equation*}

It is shown that the only bounded solutions of $\tilde{L}(u)=0$ in $\mathbb{R}^2$ are precisely the linear combinations of the $\varphi_k$, $k=0,1,2$, see Baraket and Pacard's paper \cite{Baraket1997} for a detailed proof. Let us define $\varphi_{i,j}\left(\frac{y}{\lambda}\right):=\varphi_i\left(\frac{\Pi_{\xi_j}(y)}{\lambda}\right)$ as a function on $\mathbb{S}^2_{\lambda}$ without ambiguity, where $i=0,1,2$, $j=1,2,3,4$ and $y\in\mathbb{S}^2$.

Moreover, let us pick a large but fixed number $R_1>0$. We introduce another type of cut-off functions:
\begin{equation*}\label{eq:cutoff2}
\chi_{R}(s)=1 \textrm{ for }s\leq R; \chi_{R}(s)=0 \textrm{ for } s\geq R+1; 0<\chi_{R}<1 \textrm{ for } R<s<R+1.
\end{equation*}
We further assume that 
\begin{equation*}
\abs{\chi'_{R}(s)}\leq 2.
\end{equation*}
Let us denote $\chi_{R,k}(\frac{y}{\lambda})=\chi_{R}(\frac{\abs{\Pi_{\xi_k}(y)}}{\lambda})$, $k=1,2,3,4$.

Then, let us introduce some functional set-ups of the problem.

Let
\begin{equation*}
L^{p}_{s}(\mathbb{S}^2_{\lambda})=\big\{u\in L^{p}(\mathbb{S}^2_{\lambda})\rvert u\left(\frac{y}{\lambda}\right)=u\left(\frac{Ty}{\lambda}\right) \textrm{ for all }T\in T_d\big\},
\end{equation*}
where $1\leq p\leq \infty$.

We consider the following norms
\begin{equation*}
\norm{\psi}_{\infty}=\sup\limits_{\frac{y}{\lambda}\in \mathbb{S}^2_{\lambda}} \big\vert\psi\left(\frac{y}{\lambda}\right)\big\vert, \textrm{ } \norm{\psi}_{\ast}=\sup\limits_{\frac{y}{\lambda}\in \mathbb{S}^2_{\lambda}}\left(\sum\limits_{j=1}^4 \left(1+\frac{|\Pi_{\xi_j}(y)|}{\lambda}\right)^{-3}+\lambda^2\right)^{-1}|\psi\left(\frac{y}{\lambda}\right)|.
\end{equation*}

Let
\begin{equation*}
\mathcal{C}=\big\{u\in L^{\infty}(\mathbb{S}^2_{\lambda})\rvert u\left(\frac{y}{\lambda}\right)=u\left(\frac{Ty}{\lambda}\right),  \textrm{ for all }T\in T_d \textrm{ and } \norm{u}_{\ast}<\infty \big\}.
\end{equation*}

Let
\begin{equation*}
\mathcal{C}_{\ast}=\big\{u\in L^{\infty}(\mathbb{S}^2_{\lambda})\rvert u\left(\frac{y}{\lambda}\right)=u\left(\frac{Ty}{\lambda}\right),  \textrm{ for all }T\in T_d, \norm{u}_{\ast}<\infty \textrm{ and }u\perp \varphi_{0,j}\chi_{R_1,j}\big\}.
\end{equation*}

 Given $h\in \mathcal{C}$, we consider the linear problem of finding a function $\phi\in \mathcal{C}_{\ast}$ and scalars $c_j$, $j=1,2,3,4$ such that
\begin{equation}\label{eq:linearp1}
L(\phi)=h+\sum\limits_{j=1}^4 c_j \chi_{R_1,j}\varphi_{0,j} \textrm{ in }\mathbb{S}^2_{\lambda}.
\end{equation}
We observe that the orthogonality condition in the problem above is only taken with respect the approximate kernel due to the dilations. Furthermore, we can easily find that the elements in $\mathcal{C}_{\ast}$ are also perpendicular to the approximate kernels that are generated by translations, i.e.
\begin{equation*}
u\perp \varphi_{i,j}\chi_{R_1,j}, \textrm{ for all }i=0,1,2 \textrm{ and }j=1,2,3,4, \textrm{  }u\in \mathcal{C}_{\ast}.
\end{equation*}

Our main result in this section states its bounded solvability, uniform in small $\lambda$ in our functional settings of the enlarged sphere $\mathbb{S}^2_{\lambda}$.
\begin{proposition}\label{prop:solvability}
There exist a positive number $\lambda_0$ and a $C$, such that for any $\lambda\in(0,\lambda_0)$, there is a unique solution to the problem (\ref{eq:linearp1}). Moreover, if $h\in C^{\alpha}(\mathbb{S}^2_{\lambda})$ then 
\begin{equation}\label{eq:apriori}
\norm{\phi}_{\infty}\leq C \norm{h}_{\ast}.
\end{equation}
\end{proposition}

The proof of this result consists of two steps. The first step is to establish an uniform a priori estimate for the problem (\ref{eq:linearp1}) under the additional orthogonality conditions of $\phi$ generated by translations. More precisely, we consider the problem
\begin{equation}\label{eq:linearp2}
L(\phi)=h \textrm{ in }\mathbb{S}^2_{\lambda},
\end{equation}
\begin{equation}\label{eq:linearp3}
\int_{\mathbb{S}^2_{\lambda}}\chi_{R_1,j}\varphi_{i,j}\phi=0 \textrm{ for all }i=0,1,2, j=1,2,3,4.
\end{equation} 
\begin{lemma}\label{lm:apriori}
Assume that $h\in C^{\alpha}(\mathbb{S}^2_{\lambda})$. Then there exist positive number $\lambda_0$ and $C$, such that for any $\lambda\in(0,\lambda_0)$ and any solution to (\ref{eq:linearp2})-(\ref{eq:linearp3}), one has
\begin{equation*}
\norm{\phi}_{\infty}\leq C \norm{h}_{\ast}.
\end{equation*}
\end{lemma}

\begin{proof}
We will adopt the same technique introduced by del Pino, Kowalczyk and Musso in their paper \cite{delPino2005} to prove the invertibility of the linearized operator of the mean field equation in bounded domain but with Dirichlet boundary condition.

We prove this lemma by contradiction. We assume that there exist sequences $\lambda_n\rightarrow 0$, $h_n$ with $\norm{h_n}_{\ast}\rightarrow 0$ and $\norm{\phi_n}_{\infty}=1$ such that
\begin{equation}\label{eq:lincon1}
L(\phi_n)=h_n \textrm{ in }\mathbb{S}^2_{\lambda},
\end{equation}
\begin{equation}\label{eq:lincon2}
\int_{\mathbb{S}^2_{\lambda_n}}\chi_{R_1,j}\varphi_{i,j}\phi_n=0 \textrm{ for all }i=0,1,2, j=1,2,3,4.
\end{equation}

The contradiction is obtained via several major steps. The key step is to construct a positive supersolution in order to show that the operator $L$ satisfies the maximum principle in $\mathbb{S}^2_{\lambda}$ outside large geodesic balls centered at the points $\xi^{'}_j=\frac{\xi_j}{\lambda}$. Let us introduce the radial solution $f_0(r)=\frac{r^2-1}{r^2+1}$ in $\mathbb{R}^2$ of 
\begin{equation*}
\Delta f_0+\frac{8}{(1+r^2)^2}f_0=0.
\end{equation*}
 
We are ready to define a comparison function in $\mathbb{S}^2_{\lambda}$,
\begin{equation}\label{eq:supersolution}
\tilde{V}\left(\frac{y}{\lambda}\right)=\sum\limits_{j=1}^4 f_0\left(a\frac{|\Pi_{\xi_j}(y)|}{\lambda}\right)
\end{equation}
for $\frac{y}{\lambda}\in \mathbb{S}^2_{\lambda}$. Now let us denote $z_{\xi_j}=\frac{x_{\xi_j}}{\lambda}=\frac{\Pi_{\xi_j}(y)}{\lambda}$ for convenience.

We observe that
\begin{equation}\label{eq:superlap}
-\Delta \tilde{V}=\sum\limits_{j=1}^4 \frac{8a^2(a^2|z_{\xi_j}|^2-1)}{(1+a^2|z_{\xi_j}|^2)^3}\cdot \frac{(1+|x_{\xi_j}|^2)^2}{4}.
\end{equation}
So that for $|z_{\xi_j}|>\frac{10}{a}$ for  all $j=1,2,3,4$,
\begin{equation}\label{eq:superveri}
-\Delta \tilde{V}\geq 2\sum\limits_{j=1}^4\frac{a^2}{(1+a^2|z_{\xi_j}|)^2}\cdot \frac{(1+|x_{\xi_j}|^2)^2}{4} \geq \sum\limits_{j=1}^4 \frac{a^{-2}}{|z_{\xi_j}|^4}\cdot \frac{(1+|x_{\xi_j}|^2)^2}{4}.
\end{equation}

On the other hand, in the same region,
\begin{equation}\label{eq:superverimul}
e^{w_{\lambda}}\tilde{V}\leq C\sum\limits_{j=1}^4 \frac{1}{|z_{\xi_j}|^4}\cdot \frac{(1+|x_{\xi_j}|^2)^2}{4}.
\end{equation}
Hence if $a$ is taken small and fixed, and $R'_{2}=\frac{R_2}{\lambda}>0$ is chosen sufficiently large depending on the choice of this $a$, then we have $L(\tilde{V})<0$ in $\tilde{\mathbb{S}}^2_{\lambda}:=\lambda^{-1}\left(\mathbb{S}^2\setminus \cup^{4}_{j=1} \Pi_{\xi_j}(B(0,R_2)) \right)$. Here we are able to find a positive supersolution $\tilde{V}$ on $\tilde{\mathbb{S}}^2_{\lambda}$. Then we conclude that the operator $L$ satisfies the Maximum principle, i.e. if $L(u)\leq 0$ in $\tilde{\mathbb{S}}^2_{\lambda}$ and $u\geq 0$ on $\partial\tilde{\mathbb{S}}^2_{\lambda}$, then $u\geq 0$ in $\tilde{\mathbb{S}}^2_{\lambda}$. 

Let us fix $R_2>0$. Now let us consider the ``inner norm"
\begin{equation}\label{eq:innernorm}
\norm{\phi}_i=\sup\limits_{\bigcup\limits_{j=1}^4 \lambda^{-1}\left(\Pi_{\xi_j}(B(0,R_2))\right)}\abs{\phi}.
\end{equation}

Then the second step in this proof is to show the following claim is true: there is a constant $C$ such that if $L(\phi)=h$ in $\mathbb{S}^2_{\lambda}$ then
\begin{equation}\label{eq:innerapriori}
\norm{\phi}_{\infty}\leq C \left[\norm{\phi}_i+\norm{h}_{\ast}\right].
\end{equation}
We will need suitable barrier functions to prove the above claim.

Let $\tilde{g}_j$ be the solution of the problem
\begin{equation}\label{eq:barrier1}
-\Delta \tilde{g}_j = \frac{2}{|z_{\xi_j}|^3}+2\lambda^2 \textrm{ in }\lambda^{-1}\left(\mathbb{S}^2\setminus \Pi_{\xi_j}(B(0,R_2))\right),
\end{equation}
\begin{equation*}
\tilde{g}_j=0 \textrm{ on }\partial\left(\lambda^{-1}\Pi_{\xi_j}(B(0,R_2))\right),
\end{equation*}
for $j=1,2,3,4$.

Abuse the notation a little bit, we have
\begin{equation}\label{eq:barrieror}
-\Delta_g \tilde{g}_j =\frac{2\lambda}{|x_{\xi_j}|^3}+2 \textrm{ in } \mathbb{S}^2\setminus \Pi_{\xi_j}(B(0,R_2)),
\end{equation}
\begin{equation*}
\tilde{g}_j=0 \textrm{ on }\partial \Pi_{\xi_j}(B(0,R_2)),
\end{equation*}
in the original isothermal coordinates $x_{\xi_j}$.

By the elliptic regularity estimates, we have
\begin{equation*}
\norm{\tilde{g}_j}_{H^2(\mathbb{S}^2\setminus \Pi_{\xi_j}(B(0,R_2)))}\leq C \norm{2\lambda|x_{\xi_j}|^{-3}+2}_{L^2(\mathbb{S}\setminus \Pi_{\xi_j}(B(0,R_2)))}\leq C(\lambda+1)\leq C.
\end{equation*}

Let us introduce our barrier
\begin{equation}\label{eq:barrier2}
\tilde{\phi}=2\norm{\phi}_i \tilde{V}+\norm{h}_{\ast}\sum\limits_{j=1}^2 \tilde{g}_j.
\end{equation}
Then, it is easy to check that $L(\tilde{\phi})\leq h$ in $\tilde{\mathbb{S}}^2_{\lambda}$ and $\tilde{\phi}\geq \phi$ on $\partial\tilde{\mathbb{S}}^2_{\lambda}$. Hence, we have $\phi\leq \tilde{\phi}$ in $\tilde{\mathbb{S}}^2_{\lambda}$. Similarly, one can also show that $\phi\geq -\tilde{\phi}$ in $\tilde{\mathbb{S}}^2_{\lambda}$ and the claim follows.

In the last step, we go back to the contradiction argument. The claim in the second step shows that since $\norm{\phi_n}_{\infty}=1$, then for some $\kappa>0$, we have $\norm{\phi_n}_i \geq \kappa$. Let us set $\hat{\phi}_n(z)=\phi_n\left(\frac{\Pi^{-1}_{\xi_j}(\lambda z)}{\lambda}\right)$ where the index $j$ is such that $\sup_{|z_{\xi_j}|<R'_2}|\phi_n|\geq \kappa$. Without loss of generality, we can assume this index $j$ is the same for all  $n$. Elliptic estimates readily imply that $\hat{\phi}_n$ converges uniformly over any compact subset to a bounded solution $\hat{\phi}\neq 0$ of a problem in $\mathbb{R}^2$
\begin{equation}\label{eq:limit}
\Delta \phi + \frac{8}{(1+|z|^2)^2}\phi=0.
\end{equation}

This implies that $\hat{\phi}$ is a linear combination of the functions $\varphi_{k}$, $k=0,1,2$. However, the orthogonal conditions that all $\phi_n$'s satisfy imply that $\hat{\phi}\equiv 0$. The result of the lemma then follows from the contradiction.
\end{proof}
We are now ready to provide a complete proof of our main result of this section.
\begin{proof}[Proof of Proposition 4.1]
We first establish the validity of the a priori estimate (\ref{eq:apriori}). Lemma \ref{lm:apriori} yields
\begin{equation}\label{eq:estprj}
\norm{\phi}_{\infty}\leq C \left[\norm{h}_{\ast}+\sum\limits_{j=1}^4 |c_j| \right],
\end{equation}
hence it suffices to estimate the values of the constants $|c_j|$. Let us consider the cut-off function $\eta_{R_3,\xi_j}$ introduced in (\ref{eq:cutoff1}) for some $R_3>0$. We abuse the notation a little bit to denote $\eta_{R_3,\xi_j}$ as a function on $\mathbb{S}^2_{\lambda}$. We multiply the equation (\ref{eq:linearp1}) by the test function $\varphi_{0,j}\eta_{R_3,\xi_j}$ and integrate
\begin{equation}\label{eq:test}
\langle L(\phi), \eta_{R_3,\xi_j}\varphi_{0,j}\rangle =\langle h, \eta_{R_3,\xi_j}\varphi_{0,j}\rangle+c_j \int_{\mathbb{S}^2_{\lambda}}\chi_{R_1,j}|\varphi_{0,j}|^2.
\end{equation}

On the other hand, we have
\begin{equation}\label{eq:adjoint}
\langle L(\phi), \eta_{R_3,\xi_j}\varphi_{0,j} \rangle =\langle \phi,L(\eta_{R_3,\xi_j}\varphi_{0,j})\rangle.
\end{equation}

Now, we have
\begin{equation*}
\frac{4}{(1+|x_{\xi_j}|^2)^2}L(\eta_{R_3,\xi_j}\varphi_{0,j})=\Delta_{z_{\xi_j}} \eta_{R_3,\xi_j} \varphi_{0,j}+2\nabla_{z_{\xi_j}} \eta_{R_3,\xi_j} \nabla_{z_{\xi_j}} \varphi_{0,j}
\end{equation*}
\begin{equation*}
+ \lambda \left[O\left(\frac{r}{(1+r^2)^2}\right)+O\left(\frac{1}{(1+r^2)^2}\right)\right],
\end{equation*}
with $r=|z_{\xi_j}|$. Since $\Delta_{z_{\xi_j}} \eta_{R_3,\xi_j}= O(\lambda^2)$, $\nabla_{z_{\xi_j}} \eta_{R_3,\xi_j}= O(\lambda)$ and $\varphi_{0,j}=O(1)$, $\nabla_{z_{\xi_j}} \varphi_{0,j}=O(r^{-3})$ , we have
\begin{equation}\label{eq:errperp}
\frac{4}{(1+|x_{\xi_j}|)^2}L(\eta_{R_3,\xi_j}\varphi_{0,j})= O(\lambda^2)+ \lambda \left[O\left(\frac{r}{(1+r^2)^2}\right)+O\left(\frac{1}{(1+r^2)^2}\right)\right].
\end{equation}

Therefore, we have
\begin{equation}\label{eq:estadj}
\abs{\langle \phi, L(\eta_{R_3,\xi_j}\varphi_{0,j}) \rangle}\leq C\lambda \norm{\phi}_{\infty}.
\end{equation}

Combining this above estimate with (\ref{eq:estprj}), (\ref{eq:test}) and (\ref{eq:adjoint}), we obtain
\begin{equation*}
|c_j|\leq C\left[\norm{h}_{\ast}+\lambda \sum\limits_{k=1}^4|c_k|\right].
\end{equation*}

It follows that $|c_j|\leq C\norm{h}_{\ast}$. Furthermore, from (\ref{eq:estprj}) we know that (\ref{eq:apriori}) is true.

It only remains to verify the solvability assertion. The Fredholm alternative tells us that the problem (\ref{eq:linearp1}) has a unique solution if and only the associated homogeneous problem has only trivial solution. The homogeneous problem is equivalent as the equation (\ref{eq:linearp1}) with $h=0$. From the a priori estimate we just prove, we know that the homegeneous problem only admits trivial solution. This finishes the proof.
\end{proof}
Furthermore, if we add another orthogonal condition to $\phi$ and consider the following problem:
\begin{equation}\label{eq:linrpra1}
L(\phi)=h+\sum\limits_{j=1}^4 c_j \chi_{R_1,\xi_j}\varphi_{0,j} + c_0 \textrm{ in } \mathbb{S}^2_{\lambda},
\end{equation}
\begin{equation}\label{eq:linrpra2}
\phi\perp \varphi_{0,j}\chi_{R_1,j},
\end{equation}
\begin{equation}\label{eq:linrpra3}
\phi\perp e^{w_{\lambda}},
\end{equation}
where $h\in \mathcal{C}$, we have the following corollary:
\begin{corollary}\label{cor:linear} 
Assume that the conditions in Proposition \ref{prop:solvability} hold. The problem  (\ref{eq:linrpra1})-(\ref{eq:linrpra3}) has a unique solution. Moreover, if $h\in C^{\alpha}(\mathbb{S}^2_{\lambda})$ then
\begin{equation*}
\norm{\phi}_{\infty}\leq C\norm{h}_{\ast}.
\end{equation*}
\end{corollary}
\begin{proof}
Follow the same argument as in the proof of Proposition \ref{prop:solvability}, we test (\ref{eq:linrpra1}) with $\varphi_{0,j}\eta_{R_3,\xi_j}$:
\begin{equation}\label{eq:testpr}
\langle L(\phi), \eta_{R_3,\xi_j}\varphi_{0,j}\rangle =\langle h, \eta_{R_3,\xi_j}\varphi_{0,j}\rangle+c_j \int_{\mathbb{S}^2_{\lambda}}\chi_{R_1,j}|\varphi_{0,j}|^2+c_0\int_{\mathbb{S}^2_{\lambda}}\chi_{R_1,j}\varphi_{0,j}.
\end{equation}
Integrate (\ref{eq:linrpra1}), we have
\begin{equation}\label{eq:normpr}
\int_{\mathbb{S}^2_{\lambda}}h + \sum\limits_{j=1}^4 \int_{\mathbb{S}^2_{\lambda}}c_j \chi_{R_1,j}\varphi_{0,j}+\frac{4\pi c_0}{\lambda^2}=0.
\end{equation}

Combine (\ref{eq:testpr}) and (\ref{eq:normpr}) with Proposition \ref{prop:solvability}, we can obtain
\begin{equation*}
\frac{|c_0|}{\lambda^2}\leq C\norm{h}_{\ast}.
\end{equation*}
Hence, we have
\begin{equation*}
\norm{\phi}\leq C\left[\norm{h}_{\ast}+\frac{c_0}{\lambda^2}\right]\leq C\norm{h}_{\ast}.
\end{equation*}
\end{proof}

The result of Corollary \ref{cor:linear} implies that the unique solution $\phi=T(h)$ of the problem (\ref{eq:linrpra1})-(\ref{eq:linrpra3}) defines a continuous linear map from the Banach space $\mathcal{C}$ of all functions $h$ with certain symmetries such that $\norm{h}_{\ast}<\infty$ to $L^{\infty}_s(\mathbb{S}^2_{\lambda})$.

\section{Reduce to One Dimension}
\noindent In this section, we reduce the infinite dimensional problem of finding a $\phi$ such that
\begin{equation}\label{eq:nonlin}
S_{\rho}(w_{\lambda}+\phi)=0
\end{equation}
to a one-dimensional problem of finding appropriate scale $\lambda$ while $\rho$ is given.

We now expand $S_{\rho}(w_{\lambda}+\phi)$ as
\begin{equation}\label{eq:nonex}
S_{\rho}(w_{\lambda}+\phi)=S_{\rho}(w_{\lambda})+\mathcal{L}\left(\phi-\frac{\int_{\mathbb{S}^2} e^{w_{\lambda}}\phi}{\int_{\mathbb{S}^2}e^{w_{\lambda}}}\right)+ N(\phi),
\end{equation}
where 
\begin{equation}\label{eq:nonlinr}
N(\phi)=\left[\frac{\rho}{\int_{\mathbb{S}^2}e^{w_{\lambda}+\phi}}e^{\phi}-\frac{\rho}{\int_{\mathbb{S}^2}e^{w_{\lambda}}}-\left(\phi-\frac{\int_{\mathbb{S}^2} e^{w_{\lambda}}\phi}{\int_{\mathbb{S}^2}e^{w_{\lambda}}}\right)\right]e^{w_{\lambda}}.
\end{equation}

Since the left hand side of the equation (\ref{eq:nonlin}) is invariant if we add a constant to $\phi$, we can further assume that
\begin{equation*}
\int_{\mathbb{S}^2} e^{w_{\lambda}}\phi=0.
\end{equation*} 

We abuse the notation here to denote $\phi$ as a function in $\mathcal{C}_{\ast}$.  Moreover, we consider the problem (\ref{eq:nonlin}) in the dilated coordinates, i.e. $w_{\lambda}$, $S_{\rho}(w_{\lambda})$ and $N(\phi)$ are now considered to be functions on $\mathbb{S}^2_{\lambda}$.

To employ the reduction procedure, we shall solve the following nonlinear intermediate problem first
\begin{equation}\label{eq:nonlini1}
L(\phi)=-\lambda^2\left[S_{\rho}(w_{\lambda})+N(\phi)\right]+\sum\limits_{j=1}^4 c_j \chi_{R_1,j}\varphi_{0,j} +c_0\textrm{ in } \mathbb{S}^2_{\lambda},
\end{equation}
\begin{equation}\label{eq:nonlini2}
\phi\in \mathcal{C}_{\ast},
\end{equation}
\begin{equation}\label{eq:invconst}
\int_{\mathbb{S}^2_{\lambda}}e^{w_{\lambda}}\phi=0.
\end{equation}
We will use the solvability theory we have just established in the previous section to show the existence result of the problem (\ref{eq:nonlini1})-(\ref{eq:invconst}). We assume that the conditions in Proposition \ref{prop:solvability} hold.
\begin{lemma}\label{lm:contraction}
The problem (\ref{eq:nonlini1})-(\ref{eq:invconst}) has a unique solution $\phi$ which satisfies
\begin{equation*}
\norm{\phi}_{\infty}\leq C\lambda.
\end{equation*}
\end{lemma}
\begin{proof}
We first rewrite the problem (\ref{eq:nonlini1})-(\ref{eq:invconst}) into a fixed point form:
\begin{equation}\label{eq:fixed}
\phi= T\left(-\lambda^2\left[S_{\rho}(w_{\lambda})+N(\phi)\right]\right)\equiv A(\phi) .
\end{equation}

For some constant $C>0$ sufficiently large, let us consider the region
\begin{equation*}
\mathcal{F}\equiv \{\phi\in\mathcal{C}_{\ast}\vert \phi\perp e^{w_{\lambda}}, \textrm{ }\norm{\phi}_{\infty}\leq C\lambda\}.
\end{equation*}
From Corollary \ref{cor:linear}, we have
\begin{equation*}
\norm{A(\phi)}_{\infty}\leq C\lambda^2\left[\norm{S_{\rho}(w_{\lambda})}_{\ast}+\norm{N(\phi)}_{\ast}\right].
\end{equation*}
By Lemma \ref{lm:error}, we have the following estimate
\begin{equation*}
\norm{S_{\rho}(w_{\lambda})}_{\ast}\leq C \frac{1}{\lambda}.
\end{equation*}
Also, the definition of $N$ in (\ref{eq:nonlinr}) immediately implies that
\begin{equation*}
\lambda^2\norm{N(\phi)}_{\ast}\leq C\lambda^2 \ln{\frac{1}{\lambda}}.
\end{equation*}
It is also immediate that $N$ satisfies the contraction condition
\begin{equation*}
\lambda^2\norm{N(\phi_1)-N(\phi_2)}_{\ast}\leq C\norm{\phi_1^2-\phi^2_2}_{\infty}+ C\lambda\norm{\phi_1-\phi_2}_{\infty}\leq C\lambda \norm{\phi_1-\phi_2}_{\infty}.
\end{equation*}
Hence we get
\begin{equation*}
\norm{A(\phi)}_{\infty}\leq C\lambda,
\end{equation*}
\begin{equation*}
\norm{A(\phi_1)-A(\phi_2)}_{\infty}\leq C\lambda\norm{\phi_1-\phi_2}_{\infty},
\end{equation*}
for sufficiently small $\lambda$.

Therefore, the operator $A$ is a contraction mapping of $\mathcal{F}$ if $\lambda\in (0,\lambda_0)$ where $\lambda_0$ is a constant small enough. The existence of a unique fixed point is guaranteed. This concludes the proof.
\end{proof}

\begin{lemma}\label{lm:consts}
For all $\phi$ found in Lemma \ref{lm:contraction}, we have
\begin{equation*}
c_j=c, \textrm{ }c_0=-\frac{\lambda^2}{\pi} \mathcal{A} c ,
\end{equation*}
for some constant $c$, where $c_j$'s, $j=1,2,3,4$ are coefficients in (\ref{eq:nonlini1})
 and 
\begin{equation*} 
 \mathcal{A}=\int_{\mathbb{R}^2}\chi_{R_1}\varphi_0 \frac{4}{(1+\lambda^2|z|^2)^2}dz.
 \end{equation*}
\end{lemma}
\begin{proof}
By integrating the equation (\ref{eq:nonlini1}), we have
\begin{equation}\label{eq:mean0}
\mathcal{A}\sum\limits_{j=1}^4 c_j=-\frac{4\pi c_0}{\lambda^2}.
\end{equation}

Since the problem (\ref{eq:nonlini1})-(\ref{eq:invconst}) is invariant under any orthogonal transformation $T\in T_d$, we have
\begin{equation}\label{eq:invc}
\langle L(\phi), \eta_{R_3,\xi_j}\varphi_{0,j}\rangle=\langle L(\phi), \eta_{R_3,\xi_k}\varphi_{0,k}\rangle, 
\end{equation}
for any $j\neq k$.

Then the lemma follows if we combine (\ref{eq:mean0}) and (\ref{eq:invc}).
\end{proof}

We also need to estimate the dependence of $\phi$ as a function of $\mathbb{S}^2$ on the parameter $\lambda$.
\begin{lemma}\label{lm:depenl}
The fixed point $\phi$ found in Lemma \ref{lm:contraction} satisfies
\begin{equation*}
\norm{\frac{\partial \phi}{\partial\lambda}}_{\infty}\leq C.
\end{equation*}
\end{lemma}
\begin{proof}
We study the problem (\ref{eq:nonlini1})-(\ref{eq:invconst}) on $\mathbb{S}^2$:
\begin{equation*}
\mathcal{L}(\phi)=-\left[S_{\rho}(w_{\lambda})+N(\phi)\right]+\sum\limits_{j=1}^4 \frac{c_j}{\lambda^2}\chi_{R_1,j}(\frac{y}{\lambda})\varphi_0 (\frac{\Pi_{\xi_j}(y)}{\lambda})+\frac{c_0}{\lambda^2}\textrm{ for }y\in \mathbb{S}^2,
\end{equation*}
\begin{equation*}
\int_{\mathbb{S}^2} \phi \chi_{R_1,j}(\frac{y}{\lambda})\varphi_{0}(\frac{\Pi_{\xi_j}(y)}{\lambda})=0,
\end{equation*}
\begin{equation*}
\int_{\mathbb{S}^2}e^{w_{\lambda}}\phi=0,
\end{equation*}
and $\phi$ is invariant under any orthogonal transformation $T\in T_d$. We differentiate the above equation with respect to $\lambda$:
\begin{equation*}
\mathcal{L}(\frac{\partial\phi}{\partial\lambda})+ \frac{\partial (\frac{\rho}{\int_{\mathbb{S}^2}e^{w_{\lambda}}}e^{w_{\lambda}})}{\partial\lambda}\phi=-\left[\frac{\partial S_{\rho}(w_{\lambda})}{\partial\lambda}+\frac{\partial N(\phi)}{\partial\lambda}\right]+ \sum\limits_{j=1}^4 \frac{\partial c'_j}{\partial\lambda}\chi_{R_1,j}(\frac{y}{\lambda})\varphi_{0}(\frac{\Pi_{\xi_j}(y)}{\lambda})
\end{equation*}
\begin{equation*}
+\sum\limits_{j=1}^4 c'_j(-\frac{|\Pi_{\xi_j}(y)|}{\lambda^2}) \chi'_{R_1}(\frac{|\Pi_{\xi_j}(y)|}{\lambda})\varphi_0 (\frac{\Pi_{\xi_j}(y)}{\lambda})+ \sum\limits_{j=1}^4  c'_j\chi_{R_1,j}\frac{\partial \varphi_0(\frac{\Pi_{\xi_j}(y)}{\lambda})}{\partial\lambda}+ \frac{\partial c'_0}{\partial\lambda},
\end{equation*}
where $c'_j=\frac{c_j}{\lambda^2}$, for $j=0,\cdots,4$.

Again we blow up the sphere to $\mathbb{S}^2_{\lambda}$, then we have
\begin{equation}\label{eq:dependprb}
L(\frac{\partial\phi}{\partial\lambda})=-\lambda^2\left[\frac{\partial (\frac{\rho}{\int_{\mathbb{S}^2}e^{w_{\lambda}}}e^{w_{\lambda}})}{\partial\lambda}\phi+\frac{\partial S_{\rho}(w_{\lambda})}{\partial\lambda}+\frac{\partial N(\phi)}{\partial\lambda}-\frac{\partial c'_0}{\partial\lambda}\right]
\end{equation} 
\begin{equation*}
+\sum\limits_{j=1}^4 c'_j(-|\Pi_{\xi_j}(y)|) \chi'_{R_1}(\frac{|\Pi_{\xi_j}(y)|}{\lambda})\varphi_{0,j}+\sum\limits_{j=1}^4 \lambda^2 c'_j \chi_{R_1,j}\frac{\partial \varphi_{0,j}}{\partial\lambda}
\end{equation*}
\begin{equation*}
+\sum\limits_{j=1}^4 \lambda^2\frac{\partial c'_j}{\partial\lambda}\chi_{R_1,j}\varphi_{0,j},
\end{equation*}
and
\begin{equation*}
\int_{\mathbb{S}^2_{\lambda}} \frac{\partial\phi}{\partial\lambda}\chi_{R_1,j}\varphi_{0,j}=-\int_{\mathbb{S}^2_{\lambda}}\phi\left[(-\frac{|\Pi_{\xi_j}(y)|}{\lambda^2})\chi'_{R_1}(\frac{|\Pi_{\xi_j}(y)|}{\lambda})\varphi_{0,j}+ \chi_{R_1,j}\frac{\partial\varphi_{0,j}}{\partial\lambda}\right].
\end{equation*}
Employ the same argument in the proof of Proposition \ref{prop:solvability}, we have
\begin{equation*}
|c_j|\leq C\lambda.
\end{equation*}
Therefore, by integrating (\ref{eq:dependprb}), we have
\begin{equation}\label{eq:const1}
|\frac{\partial c'_0}{\partial\lambda}|\leq C, \textrm{  }|\frac{\partial c'_j}{\partial\lambda}|\leq \frac{C}{\lambda^2}.
\end{equation}

Furthermore, calculation shows that
\begin{equation}\label{eq:exp1}
\lambda^2\norm{\frac{\partial (\frac{\rho}{\int_{\mathbb{S}^2}e^{w_{\lambda}}}e^{w_{\lambda}})}{\partial\lambda}\phi}_{\ast}\leq C,
\end{equation}
\begin{equation}\label{eq:error1}
\lambda^2\norm{\frac{\partial S_{\rho}(w_{\lambda})}{\partial\lambda}}_{\ast}\leq C.
\end{equation}
Note that here we use Lemma \ref{lm:deriv} to derive the above two estimates.

It is also easy to check that
\begin{equation}\label{eq:const2}
\norm{\sum\limits_{j=1}^4 c'_j(-|x_{\xi_j}|) \chi'_{R_1}(|z_{\xi_j}|)\varphi_{0,j}+\sum\limits_{j=1}^4 c_j \chi_{R_1,j}\frac{\partial \varphi_{0,j}}{\partial\lambda}}_{\ast}\leq C.
\end{equation}

We now use the orthogonal condition
\begin{equation*}
\int_{\mathbb{S}^2_{\lambda}}e^{w_{\lambda}}\phi=0
\end{equation*}
together with (\ref{eq:expinner}) and (\ref{eq:exp1}) to derive that
\begin{equation}\label{eq:hot1}
\lambda^2\norm{\frac{\partial N(\phi)}{\partial\lambda}}_{\ast}\leq C+ C\lambda \norm{\frac{\partial \phi}{\partial\lambda}}_{\infty}.
\end{equation}

We set $b_j$ as follows
\begin{equation*}
b_j\int_{\mathbb{S}^2_{\lambda}}\chi_{R_1,j}|\varphi_{0,j}|^2=\int_{\mathbb{S}^2_{\lambda}}\phi\left[(-\frac{|z_{\xi_j}|}{\lambda})\chi'_{R_1}(|z_{\xi_j}|)\varphi_{0,j}+ \chi_{R_1,j}\frac{\partial\varphi_{0,j}}{\partial\lambda}\right].
\end{equation*}
We can easily verify that
\begin{equation}\label{eq:constb}
|b_j|\leq C.
\end{equation}
Consider the function $\tilde{h}$ defined as follows
\begin{equation*}
\tilde{h}=-\lambda^2\left[\frac{\partial (\frac{\rho}{\int_{\mathbb{S}^2}e^{w_{\lambda}}}e^{w_{\lambda}})}{\partial\lambda}\phi+\frac{\partial S_{\rho}(w_{\lambda})}{\partial\lambda}+\frac{\partial N(\phi)}{\partial\lambda}-\frac{\partial c'_0}{\partial\lambda}\right]
\end{equation*}
\begin{equation*}
+\sum\limits_{j=1}^4 c'_j(-|x_{\xi_j}|) \chi'_{R_1}(|z_{\xi_j}|)\varphi_{0,j}+\sum\limits_{j=1}^4 \lambda^2 c'_j \chi_{R_1,j}\frac{\partial \varphi_{0,j}}{\partial\lambda}
\end{equation*}
\begin{equation*}
-\sum\limits_{j=1}^4 b_j L(\eta_{R_3,\xi_j}\varphi_{0,j}).
\end{equation*}

If $\tilde{\psi}$ is the unique solution to the following problem:
\begin{equation*}
L(\tilde{\psi})=\tilde{h}+\sum\limits_{j=1}^4 d_j \chi_{R_1,j}\varphi_{0,j},
\end{equation*}
\begin{equation*}
\tilde{\psi}\perp \chi_{R_1,j}\varphi_{0,j},
\end{equation*}
\begin{equation*}
\tilde{\psi}\in \mathcal{C}_{\ast}.
\end{equation*}

Then, we can express $\frac{\partial\phi}{\partial\lambda}$ in terms of $\tilde{\psi}$, i.e.
\begin{equation}\label{eq:normalize}
\frac{\partial\phi}{\partial\lambda}=\tilde{\psi}+\sum\limits_{j=1}^4 b_j \eta_{R_3,\xi_j}\varphi_{0,j}.
\end{equation}

Finally, combine (\ref{eq:const1})-(\ref{eq:normalize}) and (\ref{eq:errperp}),  then apply Lemma \ref{lm:apriori}, we have
\begin{equation*}
\norm{\frac{\partial\phi}{\partial\lambda}}_{\infty}\leq C.
\end{equation*}
\end{proof}
\section{Solving the reduced problem}
\noindent In this section, we now turn to solve
\begin{equation*}
S_{\rho}(w_{\lambda}+\phi)=0.
\end{equation*}
\begin{lemma}\label{lm:energerr}
We calculate the energy of the $w_{\lambda}+\phi$ 
\begin{equation*}
J_{\rho}(w_{\lambda}+\phi)=J_{\rho}(w_{\lambda})+ O(\lambda^2),
\end{equation*}
where $\phi$ is found through the fixed point argument in Section 5.
\end{lemma}
\begin{proof}
Expanding $J_{\rho}(w_{\lambda}+\phi)$ yields
\begin{equation}\label{eq:erre}
J_{\rho}(w_{\lambda}+\phi)=J_{\rho}(w_{\lambda})+\langle S_{\rho}(w_{\lambda}+\theta \phi),\phi\rangle_{\mathbb{S}^2},
\end{equation}
for some $\theta\in (0,1)$.

Let us try to estimate $S_{\rho}(w_{\lambda}+\theta\phi)$:
\begin{equation}\label{eq:erreq}
S_{\rho}(w_{\lambda}+\theta\phi)=S_{\rho}(w_{\lambda})+ \theta \Delta\phi+O(\lambda e^{w_{\lambda}}).
\end{equation}

By the fact that $\norm{\phi}_{\infty}\leq C\lambda$ and Lemma \ref{lm:error}, we have
\begin{equation*}
\langle S_{\rho}(w_{\lambda}),\phi\rangle_{\mathbb{S}^2}=o(\lambda^2).
\end{equation*}

It is easy to check that
\begin{equation*}
\int_{\mathbb{S}^2}|e^{w_{\lambda}}\phi|\leq C\lambda.
\end{equation*}

We only need to estimate the inner product of $\phi$ and the remaining term in (\ref{eq:erreq}):
\begin{eqnarray*}
|\langle \Delta\phi, \phi\rangle_{\mathbb{S}^2}|&=&|\langle\mathcal{L}(\phi),\phi\rangle_{\mathbb{S}^2}-\frac{\rho}{\int_{\mathbb{S}^2}e^{w_{\lambda}}}\langle e^{w_{\lambda}}\phi,\phi \rangle_{\mathbb{S}^2}|\\
&=&|\langle\mathcal{L}(\phi),\phi\rangle_{\mathbb{S}^2}|+ O(\lambda^2)\\
&=&|\sum\limits_{j=1}^4 c'_j\langle \chi_{R_1}(|z_{\xi_j}|)\varphi_0(z_{\xi_j}),\phi\rangle_{\mathbb{S}^2}+c'_0 \int_{\mathbb{S}^2}\phi|+O(\lambda^2)\\
&=&O(\lambda^2).
\end{eqnarray*}
Therefore, we have
\begin{equation*}
J_{\rho}(w_{\lambda}+\phi)= J_{\rho}(w_{\lambda})+O(\lambda^2).
\end{equation*}
\end{proof}
If we consider $J_{\rho}(w_{\lambda}+\phi)$ as a function of $\lambda$, then Lemma \ref{lm:energyest} and Lemma \ref{lm:energerr} imply that
\begin{equation*}
J_{\rho}(w_{\lambda})=-64\pi^2 \sum\limits_{j< k} G(\xi_j,\xi_k)-32\pi\ln{(4\pi)}+2\epsilon \ln{\lambda}
\end{equation*}
\begin{equation*}
+384\pi \lambda^2\ln{\lambda}-\epsilon\left(\ln{(\pi)}-4\pi \sum\limits_{j< k}G(\xi_j,\xi_k)\right)+O(\lambda^2).
\end{equation*}
By the standard degree theory, we have the following lemma concerning the critical point of $J_{\rho}(w_{\lambda}+\phi)$:
\begin{lemma}\label{lm:critical}
The energy $J_{\rho}(w_{\lambda}+\phi)$ is a $C^1$ function with respect to $\lambda$ for $\lambda\in (\lambda_1,\lambda_2)$.  Then, there exists a local maximum point $\lambda_{\ast}$ of the function $J_{\rho}(w_{\lambda}+\phi)$. Furthermore, we have
\begin{equation*}
\epsilon=(384\pi+o(1))\lambda^2_{\ast}\ln{\frac{1}{\lambda_{\ast}}}, \textrm{ as }\epsilon\rightarrow 0,
\end{equation*}
where $\rho=32\pi+\epsilon$ and $\epsilon\in (0,\epsilon_0)$.
\end{lemma}
Finally, we conclude the proof of Theorem {\ref{th:mainresult}} by the following lemma.
\begin{lemma}\label{lm:conc}
When $\lambda=\lambda_{\ast}$, we have
\begin{equation*}
c'_j=0
\end{equation*}
for $j=0,\cdots,4$, where
\begin{equation*}
S_{\rho}(w_{\lambda}+\phi)=\sum\limits_{j=1}^4 c'_j \chi_{R_1}(|z_{\xi_j}|)\varphi_{0}(z_{\xi_j})+c'_0.
\end{equation*}
\end{lemma}
\begin{proof}
Since $\lambda_{\ast}$ is a critical point of the function $J_{\rho}(w_{\lambda}+\phi)$, we have
\begin{equation*}
\frac{\partial J_{\rho}(w_{\lambda}+\phi)}{\partial\lambda}\big\vert_{\lambda=\lambda_{\ast}}=\langle S_{\rho}(w_{\lambda}+\phi), \frac{\partial (w_{\lambda}+\phi)}{\partial\lambda}\rangle_{\mathbb{S}^2}\bigg\vert_{\lambda=\lambda_{\ast}}=0. 
\end{equation*}
Computation shows that:
\begin{equation*}
\langle S_{\rho}(w_{\lambda}+\phi), \frac{\partial (w_{\lambda}+\phi)}{\partial\lambda}\rangle_{\mathbb{S}^2}=\int_{\mathbb{S}^2}\left[\sum\limits_{j=1}^4 c'_j \chi_{R_1}(|z_{\xi_j}|)\varphi_0(z_{\xi_j})+ c'_0\right]\left(\frac{\partial w_{\lambda}}{\partial\lambda}+\frac{\partial\phi}{\partial\lambda}\right)
\end{equation*}
\begin{equation*}
=\sum\limits_{j=1}^4 c_j\int_{\mathbb{S}^2_{\lambda}}\chi_{R_1,j}\varphi_{0,j}\left(\frac{\varphi_{0,j}}{\lambda}+O(1)\right) + c_0 \sum\limits_{j=1}^4 \int_{\mathbb{S}^2_{\lambda}} \chi_{R_1,j}\left(\frac{\varphi_{0,j}}{\lambda}+O(1)\right) 
\end{equation*}
\begin{equation*}
+c'_0\int_{\mathbb{S}^2\setminus (\bigcup\limits_{j=1}^4 \Pi_{\xi_j}(B(0,R_0)))}\left(\frac{2}{\lambda}+O(1)\right)dx
\end{equation*}
\begin{equation*}
=\left(\frac{\mathcal{B}}{\lambda}+O(1)\right)\sum\limits_{j=1}^4 c_j + \left(\frac{\mathcal{A}}{\lambda}+O(1)\right) c_0 +   \left(\frac{8(\pi-\mathcal{C})}{\lambda}+O(1)\right) c'_0,
\end{equation*}
where $\mathcal{B}$ and $\mathcal{C}$ are the following constants 
\begin{equation*}
\mathcal{B}=\int_{\mathbb{R}^2}\chi_{R_1}|\varphi_0|^2 \frac{4}{(1+\lambda^2|z|^2)^2}dz,
\end{equation*}
\begin{equation*}
\mathcal{C}=\int_{B(0,R_0)}\frac{4}{(1+|x|^2)^2}dx.
\end{equation*}

We know from Lemma \ref{lm:consts} and the above calculations that
\begin{equation*}
\left(\frac{\mathcal{B}}{\lambda}-\frac{2\mathcal{A}(1-\frac{\mathcal{C}}{\pi})}{\lambda}+ O(1)\right)c=0.
\end{equation*}

It is easy to see that by choosing $R_1$ sufficiently large, we have
\begin{equation*}
\mathcal{B}-2\mathcal{A}(1-\frac{\mathcal{C}}{\pi})\neq 0.
\end{equation*}

Therefore, we have
\begin{equation*}
c_j=0
\end{equation*}
for all $j=0,\cdots,4$.

Finally, we get the $\phi_{\ast}$ associated to $\lambda_{\ast}$ such that
\begin{equation*}
S_{\rho}(w_{\lambda_{\ast}}+\phi_{\ast})=0.
\end{equation*}

The exact blow-up solution  $w_{\lambda_{\ast}}+\phi_{\ast}$ of the equation (\ref{eq:mfeS2}) is found.
\end{proof}
\nocite{*}
\bibliographystyle{plain}
\bibliography{bumfes2}
\end{document}